\documentclass[12pt]{article}
\usepackage[a4paper, margin=2.3cm]{geometry}
\usepackage{amssymb,amsthm, amsmath}
\usepackage{hyperref}
\usepackage{color}

\newtheorem{theorem}{Theorem}[section]
\newtheorem{corollary}[theorem]{Corollary}
\newtheorem{lemma}[theorem]{Lemma}
\newtheorem{proposition}[theorem]{Proposition}
\newtheorem{conjecture}[theorem]{Conjecture}
\newtheorem{question}[theorem]{Question}

\theoremstyle{definition}

\newcommand{\R}{\mathbb R}
\newcommand{\C}{\mathbb C}

\author{Frank de Zeeuw}
\title{Ordinary lines in space}

\begin{document}
\date{}
\maketitle

\begin{abstract}
We prove that if a finite point set in $\R^3$ does not have too many points on a plane, then it spans a quadratic number of ordinary lines.
This answers the real case of a question of Basit, Dvir, Saraf, and Wolf.
It shows that there is a significant difference in terms of ordinary lines between planar point sets, which may span a linear number of ordinary lines, and truly three-dimensional point sets.
Our proof uses a projection argument of Kelly combined with a theorem of Beck on the number of spanned lines of a planar point set.
\end{abstract}

\section{Introduction}

Given a finite point set, an \emph{ordinary line} for that set is a line that contains exactly two points from the set.
The following theorem is a classic in combinatorial geometry.

\begin{theorem}[Sylvester--Gallai]\label{thm:sylvestergallai}
If a finite point set in $\R^2$ is not contained in a line,
then it spans an ordinary line.
\end{theorem}

Dirac and Motzkin conjectured a quantitative version, which states that any set of $n$ points spans at least $n/2$ points,
unless the points are collinear or $n=7$ or $n=13$.
This was proved for sufficiently large point sets by Green and Tao \cite{GT}.

\begin{theorem}[Green--Tao]\label{thm:greentao}
There is an $n_0$ such that the following holds for $n\geq n_0$.
If a set of $n$ points in $\R^2$ is not contained in a line,
then it spans at least $n/2$ ordinary lines.
\end{theorem}

This bound is best possible for general $n$, 
because there are constructions due to B\"or\"oczky that have this number of ordinary lines.
Specifically, for even $n$ we can place $n/2$ points on a conic and $n/2$ points on a line in such a way that that there are exactly $n/2$ ordinary lines (see \cite{GT} for details).

Various higher-dimensional variants of the Sylvester--Gallai theorem have been considered, for instance for planes in space,
but not much attention has been paid to ordinary lines in space.
One reason for this may be that both the Sylvester--Gallai theorem and the Green--Tao theorem hold word-for-word in $\R^3$ (by a generic projection argument, like that in the proof of Corollary \ref{cor:beck3D} below).
The bound $n/2$ in the Green--Tao theorem is also best possible in space, by the same construction on a conic and a line.
But note that such a construction only seems to work if the conic and the line are contained in the same plane.

Basit, Dvir, Saraf, and Wolf \cite{BDSW} observed that, indeed, the minimum number of ordinary lines is different for sets contained in a plane and sets not contained in a plane.
They proved that, for $n\geq 24$ points in $\R^3$ with at most $n-2$ points on a plane, there are at least $3n/2$ ordinary lines.
Their proof uses rank bounds for design matrices and works in $\C^3$.
They suggest that, if one assumes that no $cn$ points are coplanar for some $c<1$, then a better bound should hold, perhaps even a quadratic one.
The goal of this paper is to prove such a quadratic bound in $\R^3$.

\newpage
Our main result is the following.

\begin{theorem}\label{thm:main}
For every $\alpha <1$ there is a $c_\alpha>0$ such that,
if a set $P$ of $n$ points in $\R^3$ has at most $\alpha n $ points on any plane,
then $P$ spans at least $c_\alpha n^2$ ordinary lines.
\end{theorem}

The best construction that we can think of 
consists of $n/2$ points on each of two skew lines, and spans $n^2/4$ ordinary lines.
The constants $c_\alpha$ that we obtain are far smaller than $1/4$, so it remains an open problem to determine optimal values. 
We discuss the constants in Theorem \ref{thm:main}
and the possibility of improving them in Section \ref{sec:constants}.

It is natural to ask what happens for point sets with more than $\alpha n$ points coplanar.
Our theorem, together with Theorem \ref{thm:greentao},
leads to a detailed answer, which we outline in Section \ref{sec:almostcoplanar}.
In short, if at most $n-k$ points are coplanar for some constant $k$, then for sufficiently large $n$ there are at least $(k+1/2)n-O(k^2)$ ordinary lines, and this is best possible.

As the work of Basit et al. \cite{BDSW} suggests,
the statement of Theorem \ref{thm:main} may also be true in $\C^3$, 
but we have not been able to prove this.
Much of our proof does carry over to the complex case,
except for one crucial ingredient, 
the Sylvester--Gallai theorem.
In Section \ref{sec:complex} we state two conjectures that would allow us to complete our proof in $\C^3$.

\section{Spanned lines in \texorpdfstring{$\R^2$}{the plane}}

A crucial tool in the proof of our theorem is the following result of Beck \cite[Theorem 3.1]{Be}.

\begin{theorem}[Beck]\label{thm:beck}
There are constants $\beta,\gamma >0$ such that the following holds.
If a set $P$ of $n$ points in $\R^2$
has at most $\beta n$ points on a line, 
then it spans at least $\gamma n^2$ lines.
\end{theorem}

Beck does not give an explicit value for $\beta$ and $\gamma$, and his proof gives very small values.  
Using an inequality of Langer \cite{La}, which is based on some sophisticated algebraic geometry, one can obtain $\beta = 2/3 $ and $\gamma = 1/9$; see \cite{DZ}.
To keep our result independent of the more advanced theory in \cite{La}, we will not use these values in our proof, but in Section \ref{sec:constants} we will discuss what they would give.

We will need a version of Theorem \ref{thm:beck} in space, which follows by a standard generic projection argument.
Note that we use the same constants $\beta$ and $\gamma$ in Corollary \ref{cor:beck3D} as in Theorem \ref{thm:beck}, and we will refer to these constants in the proof of our main theorem.

\begin{corollary}\label{cor:beck3D}
There are constants $\beta,\gamma >0$ such that the following holds.
If a set $P$ of $n$ points in $\R^3$ 
has at most $\beta n$ points on a line, 
then it spans at least $\gamma n^2$ lines.
\end{corollary}
\begin{proof}
We choose a point $q$ that does not lie on any plane spanned by $P$, 
and we choose a plane $\pi$ so that the plane through $q$ parallel to $\pi$ contains no point of $P$.
We project $P$ from $q$ onto $\pi$, i.e., we map a point $p\in P$ to the point on $\pi$ where the line through $p$ and $q$ meets $\pi$.
Let $Q\subset \pi$ be the image of $P$.
Because of the way we chose $q$ and $\pi$, the projection is a bijection between $P$ and $Q$, 
and it preserves collinearity (in the sense that three points of $P$ are collinear if and only if their projections are collinear).
Thus we have  $|Q| = |P|=n$,
and $Q$ also has at most $\beta n$ points collinear.
By Theorem \ref{thm:beck} in the plane $\pi$,
$Q$ spans at least $\beta n^2$ lines, 
and since the projection preserves collinearity, 
$P$ also spans that many.
\end{proof}

We also need a version of Beck's theorem for sets with a larger number of collinear points, which Beck \cite[Theorem 1.2]{Be} deduced from Theorem \ref{thm:beck}.
We give a proof for the sake of completeness, and because it bears some similarity to a part of the proof of our main theorem (see Proposition \ref{prop:manycoplanar}).

\begin{corollary}\label{cor:flexbeck}
For every $\beta'<1$ there is a $\gamma'>0$ such that the following holds.
If a set $P$ of $n$ points in $\R^2$ or $\R^3$
has at most $\beta' n$ points on a line, 
then it spans at least $\gamma' n^2$ lines.
\end{corollary}
\begin{proof}
We prove this in $\R^2$, and it follows in $\R^3$ by the same generic projection argument as in the previous corollary.
Suppose that the maximum number of points of $P$ on a line is $\lambda \lvert P\rvert$ for $\lambda \leq \beta'$, and let $\ell$ be a line with $\lambda \lvert P\rvert$ points of $P$.
We can assume that $\lambda > \beta$, since otherwise we are done by Theorem \ref{thm:beck} with $\gamma'=\gamma$.
Set $P' = P\backslash \ell$, so $|P'| = (1-\lambda)\lvert P\rvert$.
If $P'$ has at most $\beta|P'|$ points collinear, then by Theorem \ref{thm:beck} it spans at least
\[\gamma|P'|^2 = \gamma(1-\lambda)^2\lvert P\rvert^2 \geq \gamma(1-\beta')^2 \lvert P\rvert^2\] 
lines,
so we are done with $\gamma' = \gamma(1-\beta')^2$.
Otherwise, there is a line $\ell'$ that contains more than $\beta|P'|\geq \beta(1-\beta')\lvert P\rvert$ points of $P'$.
Then the points on $\ell$ and the points on $\ell'$ span at least 
\[(\lambda|P|-1)\cdot \beta |P'| > \frac{1}{2}\beta^2(1-\beta')\lvert P\rvert^2\]
lines, and we are done with $\gamma' = \beta^2(1-\beta')/2$.
\end{proof}

\section{Proof of the main theorem}\label{sec:proofofmain}

\subsection{Few points on a plane}

The proof of Theorem \ref{thm:main} is split into two parts.
The first part shows that there is a small $\alpha_0$ such that the theorem holds when at most $\alpha_0 n$ points are coplanar.
This then implies Theorem \ref{thm:main} for all $\alpha\leq \alpha_0$.
The second part will treat the case where more than $\alpha_0n$ points are coplanar.

Both parts of the proof use a projection argument inspired by a proof of Kelly \cite{K}.
Basically, we project the point set from one of its points onto a generic plane. By Beck's Theorem \ref{thm:beck}, that plane should contain a quadratic number of lines.
Each of those lines spans a plane together with the point that we project from, and by the Sylvester--Gallai theorem every such plane should contain an ordinary line, so we get a quadratic number of ordinary lines in the original point set.
This simplistic sketch is incorrect in several ways, but with some adjustments we will make it work.

\begin{proposition}\label{prop:fewcoplanar}
There exist $\alpha_0 >0$ and $c_{\alpha_0}>0$ such that the following holds.
If $P$ is a set of $n$ points in $\R^3$
with at most $\alpha_0 n$ coplanar,
then $P$ spans at least $c_{\alpha_0} n^2$ ordinary lines.
\end{proposition}
\begin{proof}
We set 
\[\alpha_0 = \beta\cdot\gamma~~~~~~~~\text{and}~~~~~~~~c_{\alpha_0} = \gamma^5/2,\]
where $\beta$ and $\gamma$ are the constants from Theorem~\ref{thm:beck}.
At most $\alpha_0 |P| < \beta |P|$ points of $P$ are collinear, 
so by Corollary \ref{cor:beck3D}, 
$P$ spans at least $\gamma|P|^2$ lines.
Let $P_1$ be the subset of points in $P$ that lie on at least $\gamma|P|$ spanned lines of $P$,
and let $P_2$ the subset of points on less than $\gamma|P|$ spanned lines of $P$.
Each spanned line has at least two incidences with $P$, 
so there are at least $2\gamma|P|^2$ incidences between $P$ and the spanned lines of $P$.
The points in $P_2$ contribute less than $\gamma|P|^2$ incidences, 
so the points in $P_1$ must contribute at least $\gamma|P|^2$ incidences.
Since any point contributes at most $|P|$ incidences, it follows that 
\[|P_1|\geq \gamma|P|.\]

If each point of $P_1$ lies on at least $\gamma^4 |P|$ ordinary lines of $P$, then we count at least
\[\frac{1}{2}\cdot \gamma|P|\cdot \gamma^4 |P| \geq \frac{\gamma^5}{2}|P|^2 = c_{\alpha_0}|P|^2\]
ordinary lines and we are done.
Otherwise, there is a point $p_1\in P_1$ that lies on at most $\gamma^4 |P|$ ordinary lines of $P$, as well as on at least $\gamma|P|$ spanned lines of $P$.

We project $P\backslash\{p_1\}$ from $p_1$ onto a generic plane $\pi_1$.
Specifically, we choose $\pi_1$ so that the plane through $p_1$ parallel to $\pi_1$ contains no point of $P$ other than $p_1$.
Since each of the at least $\gamma |P|$ lines through $p_1$ hits $\pi_1$ in a different point,
the image of $P$ under the projection is a set $Q_1\subset\pi_1$ with 
\[|Q_1|\geq \gamma|P|.\]
Let $Q_2$ be the set of points in $Q_1$ that correspond to ordinary lines through $p_1$ (i.e., the points in $\pi$ with a unique preimage in $P$); by assumption we have 
\[|Q_2|\leq \gamma^4 |P|.\]

Under the projection, 
the preimage of a line in $\pi_1$ is a plane containing $p_1$, 
so collinear points of $Q$ correspond to coplanar points of $P$.
Since at most $\alpha_0 |P| = \beta\gamma |P|\leq \beta|Q_1|$ points of $P$ are coplanar, 
at most $\beta|Q_1|$ points of $Q_1$ are collinear.
Thus we can apply Theorem \ref{thm:beck} to $Q_1$,
which tells us that $Q_1$ spans at least 
\[\gamma|Q_1|^2 \geq \gamma^3|P|^2\]
lines.
Among the spanned lines of $Q_1$ there are at most 
$|Q_1|\cdot |Q_2| \leq \gamma^4|P|^2$ lines that contain a point of $Q_2$.
Thus there are at least
\[\gamma^3|P|^2- \gamma^4 |P|^2 > \gamma^4|P|^2
\]
lines that contain two or more points of $Q_1$, 
but no points of $Q_2$. 
Let $L_1$ be the set of those lines.

For any line $\ell\in L_1$,
let $\sigma_\ell$ be the plane spanned by $\ell$ and $p_1$.
Then $P\cap \sigma_\ell$ contains only the preimages of $Q_1\cap \ell$,
and if $\ell\in L_1$ contains $k$ points of $Q_1$,
then $P\cap \sigma_\ell$ is contained in $k$ concurrent lines.
Moreover, 
because $\ell$ contains no points of $Q_2$, 
each of the concurrent lines in $\sigma_\ell$ contains at least two points of $P$ other than $p_1$.
This implies that $P\cap \sigma_\ell$ is not collinear,
and that none of its ordinary lines contain $p_1$.
By Theorem \ref{thm:sylvestergallai},
the plane $\sigma_\ell$ contains an ordinary line, 
which cannot pass through $p_1$.

All of the ordinary lines obtained in this way lie in different planes through $p_1$,
and they avoid $p_1$, 
so they are all distinct.
Therefore, we have found 
$\gamma^4 |P|^2> c_{\alpha_0}|P|^2$ ordinary lines for $P$.
\end{proof}

\subsection{Many points in a plane}

The second part of the proof concerns point sets that have more than $\alpha_0n$ points on a plane. 
When the point set has $\alpha n$ points on a plane with $\alpha >1/2$, it is easy to count at least
\[(1-\alpha)n\cdot (\alpha-(1-\alpha))n = (1-\alpha)(2\alpha-1) n^2\] 
ordinary lines, which proves the main theorem for $\alpha >1/2$.
This leaves a gap between $\alpha_0$ and $1/2$,
and to bridge that gap we need a more involved argument.
This argument in fact works for any $0<\alpha<1$, but only if the point set really has a plane with $\alpha n$ points. 
So we still need Proposition \ref{prop:fewcoplanar}, 
which applies also to point sets with a sublinear number of coplanar points.

First we prove a lemma that treats the situation where many of the points lie on two skew lines, i.e., two lines that are not contained in the same plane.

\begin{lemma}\label{lem:skewlines}
Let $P$ be a point set in $\R^3$,
and let $\ell$ and $\ell'$ be two \emph{skew} lines.
Then $P$ spans at least
\[ |P\cap \ell|\cdot |P\cap \ell'| - |P|\]
ordinary lines.
\end{lemma} 
\begin{proof}
Set $Q = P\backslash (\ell\cup \ell')$.
A point $p$ of $P$ on $\ell$ and a point $p'$ of $P$ on $\ell'$ span a line $\ell_{pp'}$,
which is ordinary unless it contains a point of $Q$.
But a point $q\in Q$ can lie on at most one such line $\ell_{pp'}$,
because the plane spanned by $q$ and $\ell$ contains at most one point $p'$ on $\ell'$ (since the lines are skew),
and given $p'$ there is at most one point $p$ on $\ell$ such that $p,p',q$ are collinear.
We have $|P\cap\ell|\cdot |P\cap \ell'|$ lines $\ell_{pp'}$,
and at most $|Q|\leq |P|$ of these are not ordinary,
so we get the bound in the lemma.
\end{proof}

\begin{proposition}\label{prop:manycoplanar}
For every $0<\alpha<1$ there is a $d_\alpha>0$ such that the following holds.
If $P$ is a set of $n$ points in $\R^3$ such that the maximum number of coplanar points equals $\alpha n$,
then $P$ spans at least $d_\alpha n^2$ ordinary lines.
\end{proposition}
\begin{proof}
Let $\pi$ be a plane containing $\alpha |P|$ points of $P$.
We set 
\[\mu = \alpha - \frac{1}{4}\min\{\alpha, \beta,\gamma\}\cdot (1-\alpha)^2,\]
where $\beta$ and $\gamma$ is the constant in Theorem \ref{thm:beck}.
We have $0<\mu<\alpha$.

\bigskip{\bf Case 1: Few points on a line in $\pi$.}
Suppose that at most $\mu |P|$ of the $\alpha|P|$ points of $P\cap \pi$ are collinear.
In this case we use a variant of the proof of Proposition \ref{prop:fewcoplanar}.
By Corollary \ref{cor:flexbeck} applied to $P\cap \pi$ with $\beta' = \mu/\alpha$,
there is a $\gamma'$ such that $P\cap \pi$ spans at least $\gamma' \alpha^2 \lvert P\rvert^2$ lines within the plane $\pi$.
There are at least $(1-\alpha)|P|$ points of $P$ outside $\pi$.
We can assume that there is a point $p\in P\backslash\pi$ that lies on  at most $(\gamma' \alpha^2/2)\lvert P\rvert$ ordinary lines of $P$, 
since otherwise there are at least
\[\frac{1}{2}\cdot(1-\alpha)|P|\cdot \frac{\gamma' \alpha^2}{2}\lvert P\rvert
=\frac{\gamma' \alpha^2(1-\alpha)}{4}\lvert P\rvert^2\]
 ordinary lines.

We project $P$ from $p$ to $\pi$.
Then the image of $P$ on $\pi$ includes $P\cap \pi$, so it spans at least $\gamma' \alpha^2\lvert P\rvert^2$ lines.
At most $(\gamma' \alpha^2/2)\lvert P\rvert^2$ of these lines contain points with unique preimages,
so there are at least $(\gamma' \alpha^2/2)\lvert P\rvert^2$ spanned lines in $\pi$ that contain no points with unique preimages.
As in the proof of Proposition \ref{prop:fewcoplanar},
each of these lines gives a distinct ordinary line of $P$ not containing $p$,
so we find at least $(\gamma' \alpha^2/2)\lvert P\rvert^2$ ordinary lines spanned by $P$.
Note that $\gamma'\alpha^2(1-\alpha)/4$ and $\gamma'\alpha^2/2$ depend only on $\alpha$ and (via $\gamma'$ and $\mu$) on the absolute constants $\beta$ and $\gamma$.

\bigskip{\bf Case 2: Many points on a line in $\pi$.}
Otherwise, 
there is a line $\ell\subset \pi$ such that $|P\cap \ell | = \lambda |P|$, with $\mu< \lambda \leq \alpha$.
Set $P' = P\backslash \ell$, so 
\[|P'| = (1-\lambda) |P| \geq (1-\alpha )|P|.\]
Note that any plane containing $\ell$ contains at most $(\alpha -\lambda)|P|$ points of $P'$.

\bigskip{\bf Case 2a: Many points on a line other than $\ell$.}
Suppose that more than $\beta |P'|$ points of $P'$ are collinear,
say there is a line $\ell'$ such that $|P'\cap \ell'|=\lambda'|P'|$ with $\lambda'>\beta$.
We claim that then the lines $\ell$ and $\ell'$ are skew.
Indeed, if $\ell$ and $\ell'$ were contained in the same plane,
then that plane would contain at least 
\[\lambda |P| + \lambda'|P'|  \geq (\lambda +\lambda'(1-\lambda))|P|  >(\mu + \beta(1-\alpha))|P| > \alpha |P|\]
points of $P$, 
where the last inequality holds 
because 
\[\mu =\alpha - \frac{1}{4}\min\{\alpha, \beta,\gamma\}\cdot (1-\alpha)^2 \geq \alpha - \beta(1-\alpha).\]
This would contradict the assumption that $P$ has at most $\alpha |P|$ points on a plane.

Thus $\ell$ and $\ell'$ are skew lines,
so by Lemma \ref{lem:skewlines} $P$ spans at least 
\[ |P\cap \ell|\cdot |P\cap \ell'| - |P| = \lambda\lambda' (1-\lambda)|P|^2 - |P|> \frac{1}{2}\mu\beta (1-\alpha)|P|^2\]
ordinary lines,
 and we are done.
Again note that the constant depends only on $\alpha$, $\beta$, and $\gamma$.

\bigskip{\bf Case 2b: Few points on lines other than $\ell$.}
Otherwise, $P'$ has at most $\beta|P'|$ points on any line.
Then by Corollary \ref{cor:beck3D}
$P'$ spans at least 
\[\gamma |P'|^2 \geq \gamma(1-\alpha)^2|P|^2 \]
lines,
and there are at least $\gamma|P'|$ points of $P'$ that lie on at least $\gamma|P'|$ lines spanned by $P'$.
Let $p$ be such a point.
It spans $\lambda n$ lines together with the points of $P$ on the line $\ell$,
and at most $(\alpha-\lambda)|P|$ of these lines are also spanned by $P'$,
since the plane spanned by $p$ and $\ell$ contains at most $(\alpha-\lambda)|P|$ points of $P'$.
It follows that $p$ lies on at least
\[\left(\gamma(1-\alpha)^2 +\lambda - (\alpha - \lambda)\right)|P| >
 \left(2\mu - \alpha+ \gamma(1-\alpha)^2\right) |P| \]
lines spanned by $P$.
Write 
\[\nu = 2\mu - \alpha+ \gamma(1-\alpha)^2.\]
We have $\nu > \alpha$ because 
\[2\mu-\alpha  =\alpha - \frac{1}{2}\min\{\alpha, \beta,\gamma\}\cdot (1-\alpha)^2 \geq \alpha - \gamma(1-\alpha)^2.\]
Let $P_1$ be the set of at least $\gamma\lvert P' \rvert$ points of $P'$ that lie on at least $\nu\lvert P\rvert$ lines spanned by $P$.

Let $\gamma'$ be the constant given by Corollary \ref{cor:flexbeck} for $\beta' = \alpha/\nu$.
We can assume that some point $p_1\in P_1$ lies on at most $(\gamma'/2)|P|$ ordinary lines of $P$, since otherwise we would have at least 
\[\frac{1}{2}\nu|P|\cdot \frac{\gamma'}{2}|P| > \frac{\alpha\gamma'}{4}|P|^2\]
ordinary lines.

We now finish as in the proof of Proposition \ref{prop:fewcoplanar}.
We project $P$ from $p_1$ to a set $Q_1$ on a generic plane $\pi_1$, so that $|Q_1|\geq \nu |P|$, and at most $(\gamma'/2)|P|$ points of $Q_1$ have a unique preimage.
Since $P$ has at most $\alpha |P|$ points coplanar, the image on $\pi_1$ has at most $\alpha |P|$ points  collinear.
Thus we can apply Corollary \ref{cor:flexbeck} with $\beta' = \alpha/\nu$,
which tells us that $Q_1$ spans at least $\gamma'|P|^2$ lines.
At least $(\gamma'/2)|P|^2$ of these lines contain no points with unique preimages,
so as before this leads to $(\gamma'/2)|P|^2$ ordinary lines for $P$.
Again the constant depends only on $\alpha$, $\beta$, and $\gamma$ (via $\gamma'$, $\nu$, and $\mu$).
\end{proof}

\subsection{Finishing the proof of Theorem \ref{thm:main}}

We complete the proof of Theorem \ref{thm:main} by combining Proposition \ref{prop:fewcoplanar} and Proposition \ref{prop:manycoplanar}.
Let $P$ be a set of $n$ points in $\R^3$ with at most $\alpha_1 n$ points on any plane.
If in fact $P$ has at most $\alpha_0$ points on any plane, 
where $\alpha_0$ is as in Proposition \ref{prop:fewcoplanar},
then by Proposition \ref{prop:fewcoplanar} $P$ spans at least $c_{\alpha_0}n^2$ ordinary lines.
Otherwise $P$ has more than $\alpha_0n$ points on some plane,
so there exists an $\alpha$ with $\alpha_0\leq \alpha\leq \alpha_1$, 
such that the maximum number of points of $P$ on a plane equals $\alpha n$.
Then by Proposition \ref{prop:manycoplanar} $P$ spans at least $d_{\alpha}n^2$ ordinary lines.
We can see in the proof of Proposition \ref{prop:manycoplanar} that $d_{\alpha}$ is a piecewise polynomial function of $\alpha$, 
so it attains its minimum on the interval $\alpha_0\leq \alpha\leq \alpha_1$.
If we choose
\[c_{\alpha_1} = \min\{c_{\alpha_0}, 
\min\{d_\alpha: \alpha_0\leq \alpha\leq \alpha_1\}\},\]
then we have shown that $P$ spans at least $c_{\alpha_1}n^2$ ordinary lines.
This finishes the proof.

\section{Discussion}

\subsection{Improving the constants}\label{sec:constants}
The constants that follow from our proof of Theorem \ref{thm:main} are minuscule.
For instance, in Proposition \ref{prop:fewcoplanar},
we have $\alpha_0 = \beta\cdot \gamma$ and $c_{\alpha_0} = \gamma^5/2$, where $\beta$ and $\gamma$ come from Theorem \ref{thm:beck}.
Even with the best known values of $\beta=2/3$ and $\gamma = 1/9$ (see \cite{DZ}), we get $\alpha_0=2/27$ and $c_{2/27} = 1/118098$. 
There are several ways to make this proof more efficient, but the largest $c_\alpha$ we were able to obtain without too much trouble was $c_{1/6} = 1/288$,
i.e.,
if at most $n/6$ points of $P$ are coplanar, then $P$ spans at least $n^2/288$ ordinary lines.
We do not include the proof here because the result seems to be far from optimal.

As mentioned in the introduction, the best example that we know of is $n/2$ points on each of two skew lines, 
spanning $n^2/4$ ordinary lines.
We doubt that our argument, as it is, could reach such a value.
In the proof of Proposition \ref{prop:fewcoplanar},
the largest projection image $Q_1$ that we could hope for would have size roughly $n/2$, which would follow from a proof of the strong Dirac conjecture (see for instance \cite{DZ}).
Then $Q_1$ spans at most $\binom{n/2}{2}\approx n^2/8$ lines, so if we find one ordinary line in each plane $\sigma_\ell$ spanned by these lines and $p_1$,
we would still only find $n^2/8$ ordinary lines.
This calculation is not very realistic, but it shows what the limits of the argument are.

To do better, 
we could try to use the full power of Theorem \ref{thm:greentao} of Green and Tao.
Indeed, if the line $\ell$ in the plane $\pi_1$ contains $k$ points of $Q$, then $\sigma_\ell$ should contain at least $k/2$ ordinary lines, whereas in the proof we used only one. 
However, since Theorem \ref{thm:greentao} applies only for $n\geq n_0$, we cannot use it for lines in $\pi_1$ with few points.
But even if we assume the full Dirac--Motzkin conjecture (i.e., that the statement of Theorem \ref{thm:greentao} holds for all $n$ except $n=7,13$) as well as the strong Dirac conjecture, it seems we still would not reach the bound $n^2/4$.
We leave out the details, but we calculated that this would at best give $n^2/24$ ordinary lines.

The issue with the last argument is that it should not be possible to have a minimum number of ordinary lines in many of the planes $\sigma_\ell$ simultaneously. 
Based on this, we ask the following (somewhat vague) question.

\begin{question}
Is it possible that a point set $P$ in $\R^3$ has  not too many points coplanar,
but that for many planes $\pi$ spanned by $P$ the set $P\cap \pi$ has close to $|P\cap \pi|/2$ ordinary lines?
\end{question}

Green and Tao \cite{GT} proved that if a point set in the plane spans close to the minimum number of ordinary lines, then most of its points must lie on a cubic curve, and moreover they must lie in a specific group-related configuration.
If a point set in $\R^3$ lies on a cubic surface, then it would be the case that its intersection with any plane lies on a cubic curve, 
but we doubt that it is possible that in many of these curves the points lie in a group-related configuration that spans few ordinary lines.

\subsection{Almost all points on a plane}\label{sec:almostcoplanar}

As observed in \cite{BDSW},
a point set in $\R^3$ can have a subquadratic number of ordinary lines when almost all the points are on a plane. 
Indeed,
we can place $n-k$ points on a plane with $(n-k)/2$ ordinary lines, and add $k$ points outside the plane, to get at most $k(n-k) + (n-k)/2$
ordinary lines.
More precisely, we can place the $k$ points on a line that hits the plane in one of the $n-k$ points, so that the number of ordinary lines equals $k(n-k) + (n-k)/2 - k$.
Using Theorem \ref{thm:main}, we can prove that this bound is basically tight for sufficiently large $n$.

\begin{corollary}\label{cor:almostcoplanar}
For every $k\geq 1$ there is an $n_k$ such that the following holds for all $n\geq n_k$.
If $P$ is a set of $n$ points in $\R^3$ with at most $n-k$ on any plane,
then $P$ spans at least 
\begin{equation}\label{eq:nminusk}\left(k+\frac{1}{2}\right)(n-k) - \binom{k}{2}
\end{equation}
ordinary lines.
\end{corollary}
\begin{proof}
Let $\ell\geq k$ be the smallest number so that $P$ has $n-\ell$ points on some plane $\pi$.
The points of $P$ on $\pi$ are not collinear, 
since otherwise we could find a larger $\ell$,
so by Theorem \ref{thm:greentao} we have at least $(n-\ell)/2$ ordinary lines within $\pi$, assuming $n$ is sufficiently large.
For every choice of a point of $P$ on $\pi$ and a point of $P$ off $\pi$ we get an ordinary line, unless this line is one of the at most $\binom{\ell}{2}$ lines spanned by $P\backslash\pi$.
Thus the number of ordinary lines is at least $\ell(n-\ell) - \binom{\ell}{2} +(n-\ell)/2$.
With some calculation we can check that this function is at least as large as \eqref{eq:nminusk} for $\ell$ with $k\leq\ell< n/3$.
On the other hand, 
if $\ell\geq n/3$, then we have at most $2n/3$ points of $P$ on a plane, 
so by Theorem \ref{thm:main} $P$ spans at least $c_{2/3}n^2$ ordinary lines.
For sufficiently large $n$,
$c_{2/3}n^2$ is larger than \eqref{eq:nminusk}.
\end{proof}

For instance, when at most $n-1$ points are coplanar, then Corollary \ref{cor:almostcoplanar} gives at least $3n/2-3/2$ ordinary lines for $n\geq n_1$, and this is tight,
since (for $n$ odd) we can put $n-1$ points on a plane with $(n-1)/2$ ordinary lines, and one point off the plane spanning another $n-1$ ordinary lines.
For larger $k$, 
a small gap remains between the lower bound $(k+1/2)(n-k) - \binom{k}{2}$ and the number $(k+1/2)(n-k) - k$ in the construction mentioned above.

\subsection{Discussion of the complex case}\label{sec:complex}

We briefly discuss the possibility of extending Theorem \ref{thm:main} to $\C^3$,
mainly because it lets us pose some interesting problems.
A complex equivalent would be interesting since it would be a quantitative version of the following result of Kelly \cite{K}.

\begin{theorem}[Kelly]\label{thm:kelly}
If a finite point set in $\C^3$ is not contained in a plane, 
then it spans an ordinary line.
\end{theorem}

The proof of Proposition \ref{prop:fewcoplanar} mostly carries over to $\C^3$, with the crucial exception of the Sylvester--Gallai theorem, which is false over $\C$. 
However, if we could show that the projected set $Q_1$ spans a quadratic number of lines with at most three points, then we could continue the proof by using the special configuration of the points of $P$ on the planes $\sigma_\ell$.
Specifically, if a line $\ell$ in the plane $\sigma_\ell$ contains three points of $Q_1$, 
then $P\cap \sigma_\ell$ is contained in three concurrent lines.
Thus we could apply the following lemma of Kelly and Nwankpa \cite[Theorem 3.12]{KN}, which Kelly \cite{K} used in his proof of Theorem \ref{thm:kelly}.

\begin{lemma}[Kelly--Nwankpa]\label{lem:threeconcurrent}
Let $P$ be a finite set in $\C^2$ that is not contained in one line.
If $P$ is contained in three concurrent lines,
then $P$ spans an ordinary line that does not contain the common point of the three lines.
\end{lemma}

Unfortunately, 
we have not been able to prove that $n$ points in $\C^2$, with not too many collinear, 
span a quadratic number of lines with at most three points.
For $\R^2$, such a quadratic bound can be found in \cite{DZ},
but the proof uses Melchior's equality, which does not hold in $\C^2$.

\begin{conjecture}\label{conj:atmostthree}
There exists a constant $c>0$ such that,
if a set $P$ of $n$ points in $\C^2$ has at most $cn$ collinear,
then $P$ spans at least $cn^2$ lines with at most three points.
\end{conjecture}

On the other hand, it is known that $n$ points in $\C^2$, with not too many collinear, span a quadratic number of lines with at most \emph{four} points; again see \cite{DZ}.
This could give an alternative approach to a complex equivalent of Theorem \ref{thm:main}, 
if we could prove the following variant of Lemma \ref{lem:threeconcurrent}.

\begin{conjecture}\label{conj:fourconcurrent}
Let $P$ be a finite set in $\C^2$ that is not contained in one line.
If $P$ is contained in four concurrent lines,
then $P$ spans an ordinary line that does not contain the common point of the four lines.
\end{conjecture}

\section*{Acknowledgements}
I would like to thank Gabor Tardos for pointing out a simplification in the proof of Proposition \ref{prop:fewcoplanar},
and Oliver Roche-Newton for some helpful discussions.
This work was partially supported by Swiss National Science Foundation grants 200020--162884 and 200021--175977.


\begin{thebibliography}{99}

\bibitem{BDSW}
Abdul Basit, Zeev Dvir, Shubhangi Saraf, and Charles Wolf,
\emph{On the number of ordinary lines determined by sets in complex space},
{\tt arXiv:1611.08740}, 2016.

\bibitem{Be}
J\'ozsef Beck,
\emph{On the lattice property of the plane and some problems of Dirac, Motzkin and Erd\H os in combinatorial geometry},
Combinatorica {\bf 3}, 281--297, 1983.

\bibitem{GT}
Ben Green and Terence Tao,
\emph{On sets defining few ordinary lines},
Discrete \& Computational Geometry {\bf 50}, 409--468, 2013.

\bibitem{K}
Leroy Kelly,
\emph{A resolution of the Sylvester--Gallai problem of J.-P. Serre},
Discrete \& Computational Geometry {\bf 1}, 101--104, 1986.


\bibitem{KN}
Leroy Kelly and Sonde Nwankpa,
\emph{Affine embeddings of Sylverter--Gallai designs},
Journal of Combinatorial Theory Series A {\bf 14}, 422--438, 1973.

\bibitem{La}
Adrian Langer,
\emph{Logarithmic orbifold Euler numbers of surfaces with applications},
Proceedings of the London Mathematical Society 
{\bf 86}, 358--396, 2003.


\bibitem{DZ}
Frank de Zeeuw,
\emph{Spanned lines and Langer's inequality},
{\tt arXiv:1802.08015}, 2018.

\end{thebibliography}
\end{document}